\documentclass[10pt, a4paper]{amsart}
\usepackage{amssymb}
\usepackage{amsmath}
\usepackage{youngtab}
\usepackage{young}
\usepackage{xfrac}

\usepackage{ytableau}
 \usepackage{amsfonts,mathrsfs}
\usepackage{bbm,amsmath,amssymb,amsfonts,graphicx,color,subfig,mathtools, verbatim}
   \usepackage[bookmarksopen=false,pdftex=true,breaklinks=true,%
      backref=page,pagebackref=true,plainpages=false,%
      hyperindex=true,pdfstartview=FitH,colorlinks=true,%
      pdfpagelabels=true,colorlinks=true,linkcolor=blue,%
      citecolor=red,urlcolor=green,hypertexnames=false%
      ]%
   {hyperref}
\newcommand{\K}{\mathbb{K}}
\newcommand{\N}{\mathbb{N}}

\newcommand{\R}{\mathbb{R}}
\renewcommand{\leq}{\leqslant}
\usepackage{faktor}

\newcommand{\Sym}{S}
\usepackage{MnSymbol} 

\newtheorem{theorem}{Theorem}
\newtheorem{remark}{Remark}

\newtheorem{corollary}{Corollary}
\newtheorem{proposition}{Proposition}
\newtheorem{definition}{Definition}
\newtheorem{example}{Example}
\DeclareMathOperator{\sh}{sh}
\DeclareMathOperator{\Supp}{Supp}
\DeclareMathOperator{\Mon}{Mon}
\DeclareMathOperator{\Stab}{Stab}
\DeclareMathOperator{\len}{len}
\DeclareMathOperator{\spe}{sp}

\DeclareMathOperator{\wt}{wt}
\DeclareMathOperator{\Tr}{Tr}
\DeclareFontFamily{U}{mathx}{\hyphenchar\font45}
\DeclareFontShape{U}{mathx}{m}{n}{
      <5> <6> <7> <8> <9> <10>
      <10.95> <12> <14.4> <17.28> <20.74> <24.88>
      mathx10
      }{}
\DeclareSymbolFont{mathx}{U}{mathx}{m}{n}
\DeclareFontSubstitution{U}{mathx}{m}{n}
\DeclareMathAccent{\widecheck}{0}{mathx}{"71}

\renewcommand{\leq}{\leqslant}
\renewcommand{\geq}{\geqslant}

\begin{document}
\title[Solutions to symmetric systems of equations]
{
Symmetric ideals, Specht polynomials and solutions to symmetric systems of equations
}
\author{Philippe Moustrou}
\address{Department of Mathematics and Statistics, UiT - the Arctic University of Norway, 9037 Troms\o, Norway}
\email{philippe.moustrou@uit.no}

\author{Cordian Riener}
\address{Department of Mathematics and Statistics, UiT - the Arctic University of Norway, 9037 Troms\o, Norway}
\email{cordian.riener@uit.no}

\author{Hugues Verdure}
\address{Department of Mathematics and Statistics, UiT - the Arctic University of Norway, 9037 Troms\o, Norway}
\email{hugues.verdure@uit.no}


\begin{abstract}
An ideal of polynomials is symmetric if it is closed under permutations of variables. We relate general  symmetric ideals to the so called Specht ideals generated by all Specht polynomials of a given shape. We show a connection between  the leading monomials of polynomials in the ideal and the Specht polynomials contained in the ideal. This provides applications in several contexts. Most notably,  this connection gives information about  the solutions of the corresponding set of equations. From another perspective, it restricts the isotypic decomposition of the ideal viewed as a representation of the symmetric group.

\end{abstract}

\maketitle

\section{Introduction}
Let $\Sym_n$ denote the symmetric group on $n$-elements, and let $\K$ be a field. Then $\Sym_n$ acts naturally on an $n$-dimensional space $\K^n$ by permuting coordinates. This linear action then gives rise to an action on the corresponding polynomial ring by permuting coordinates. In this article we consider  ideals $I\subset\K[X_1,\ldots X_n]$ which are stable under this action. The study of such ideals appears quite naturally in different contexts (see for example \cite{steidel2013grobner,faugere,krone2016equivariant,buse}).

Our interest in such ideals stems from algorithmic purposes: the symmetry on a set of equations often can be used to simplify its resolution. In this flavour, a fundamental result by Timofte (see \cite{timofte2003positivity,riener2012degree}) yields that  every symmetric variety defined by polynomials of  degree $d$ is non-empty over $\R$ if and only if  it contains a real point with at most $d$ distinct coordinates. Here we aim at generalising this aspect  of Timofte's result in various ways. We are able to show that - under the assumption that  the number of variables is sufficiently large - the variety corresponding to the symmetric ideal  will not contain any point with strictly more than $d$ distinct coordinates. Moreover, our result yields that the set of possible configurations of these $d$ distinct coordinates can be further restricted by the shape of the monomials of highest degree amongst the generators of  $I$, see Section~\ref{ssec:DegreePrinciple}.  The arguments put forward to establish this result are purely algebraic and work with no requirements on $\K$. In fact, this result will follow from a study of Specht polynomials contained in symmetric ideals. More precisely, we assign a partition of $n$ to these monomials and we relate these partitions  to specific Specht polynomials which belong to the ideal, see Section~\ref{sec:Main}. 

In characteristic $0$, this property also has an application to the structure of the decomposition of $I$ in terms of $\Sym_n$-representations. The action of $\Sym_n$ on $I$ turns $I$ and ${\K[X_1,\ldots,X_n]}/{I}$ into $\K[\Sym_n]$ modules. Over a field of characteristic zero this modules can be decomposed into irreducible $\K[\Sym_n]$ modules, which are usually called Specht modules and we show in Section~\ref{ssec:isotypic} that the possible Specht modules appearing in this decomposition are also very restricted (see \cite{basu2015isotypic} for a result in a  similar spirit in the real setting). One application of such a decomposition concerns sums of squares representations of positive symmetric polynomials modulo symmetric ideals. The understanding of the irreducible representations in $I$ gives a control on the complexity of sums of squares decomposition in this setup, see Section~\ref{ssec:SOS}. 

This article is structured as follows. Section~\ref{sec:prelim} collects necessary standard notations and definitions. Then, Section~\ref{sec:Specht} focuses on varieties defined by Specht polynomials and their properties. This is used in Section~\ref{sec:Main} to describe the Specht polynomials contained in symmetric ideals. Finally, Section~\ref{sec:appl} is devoted to applications.

\section{Preliminary notations and definitions}\label{sec:prelim}

\subsection*{Partitions and Young tableaux}

For any natural number $n$, one can consider its \emph{partitions}:

\begin{definition} Let $n\in\N$. A partition $\lambda$ of $n$, (denoted $\lambda \vdash n$) is a sequence $\lambda=(\lambda_1,\cdots,\lambda_l)$ of positive natural numbers ordered such that $\lambda_1\geqslant\lambda_2 \geqslant \cdots \geqslant \lambda_l \geqslant 0$ with the property  $\lambda_1+\ldots+\lambda_l = n.$ The length of $\lambda$ is 
\[\len(\lambda)= \max \{i:\ \lambda_i \neq 0\}.\]
\end{definition}

We will allow ourselves to identify partitions that only differ by $0$ terms.

\begin{definition}
For a given partition $\lambda$, its dual partition $\lambda^\perp$  is defined by \[\left(\lambda^\perp\right)_ i = \left|\{j,\ \lambda_j \geqslant i\}\right|.\]
\end{definition}

Partitions are very well known and are closely related to Young tableaux (see \textit{e.g.} \cite{S01}).

\begin{definition} Given $\lambda \vdash n$, a \emph{Young tableau} $T$ of shape  $\lambda\vdash n$, or a \emph{$\lambda$-tableau} consists of $\len(\lambda)$ rows, with
$\lambda_i$
entries in the $i$-th row.
Each entry is an element in $\{1, \ldots, n\}$, and each of these
numbers occurs exactly once.
Furthermore we write $\sh(T)= \lambda$. 
\end{definition}

\begin{definition}
Let $n\in \N$. Let  $\lambda=(\lambda_1,\cdots,\lambda_l)\vdash n$ and $\mu=(\mu_1,\cdots,\mu_m)\vdash n$ be two partitions. We say that $\lambda$ dominates $\mu$ if 
\[ \sum_{j=1}^i \lambda_j \geqslant \sum_{j=1}^i \mu_j\ \text{ holds for all }1 \leqslant i \leqslant \min\{\len(\lambda),\len(\mu)\}.\]
We will write $\lambda \trianglerighteq\mu$ in this case.
\end{definition}

Equipped with the dominance order, the set of all partitions of a given $n\in \N$  is a partially ordered set. We will further use $\lambda \not\trianglerighteq\nu$ to denote the case in which  $\lambda$ does not dominate $\nu$. Note that the order is only partial and hence this does not entail that $\nu$ dominates $\lambda$, since they also might be not comparable. Furthermore, note that $\left(\lambda^\perp\right)_1 = \len(\lambda)$, $\left(\lambda^\perp\right)^\perp = \lambda$, and  \[\lambda \trianglerighteq \mu \Leftrightarrow \mu^\perp \trianglerighteq  \lambda^\perp.\]

\subsection*{Orbit types and partitions}

The action of $\Sym_n$ on $\K^n$ naturally decomposes the space into orbits. 
\begin{definition}
For every $x\in \K^n$, the associated stabilizer subgroup $\Stab(x)\subseteq \Sym_n$ is of the form
$$\Stab(x) \simeq S_{\ell_1}\times S_{\ell_2}\times \cdots \times S_{\ell_k}$$
with $\ell_1\geq \ell_2\geq \ldots\geq \ell_k$. We hence define the \emph{orbit type} of $x$ to be  
\[
\Lambda(x):=(\ell_1, \ell_2,\cdots, \ell_k).
\]
Then, for a given $\lambda\vdash n$ we can define $$H_\lambda:=\left\{x\in \K^n\,:\, \Lambda(x)=\lambda\right\}.$$
\end{definition}
\begin{remark}
Note that we have $$\K^n=\bigcupdot_{\lambda\vdash n} H_\lambda.$$
\end{remark}

\subsection*{Polynomials, varieties, and symmetric ideals}

Let $\K$ be a field, and consider the polynomial ring in $n$ variables $\K[X_1,\cdots,X_n]$. 
For any $P\in \K[X_1,\cdots,X_n]$, we denote by $\Mon(P)$ the set of monomials appearing in $P$, and by $P_h$ its homogenous component of degree $h$. 

Given a monomial
\[
m=\prod_{j=1}^n X_j ^{k_j},
\]
its \emph{support} is 
\[
\Supp(m)=\{j,\quad k_j \neq 0 \}
\]
and its \emph{weight} $\wt(m)$ is the cardinality of its support. By taking the union over all the monomials in $\Mon(P)$, we generalize these notions to $P$ to define $\Supp(P)$ and $\wt(P)$.

To every ideal in $\K[X_1,\ldots,X_n]$ one can associate a variety:
\begin{definition}
Let $I$ be an ideal in $\K[X_1,\ldots,X_n]$. The \emph{variety} $V(I)$ associated with $I$ is the subset of $\K^n$ made by the common zeros of all the polynomials in $I$, namely
\[
V(I)= \{ x\in \K^n, \quad P(x)=0 \text{ for every } P\in I\}.
\]
\end{definition}

The action of $\Sym_n$ on $\K^n$ induces an action on $\K[X_1, \ldots, X_n]$ by permuting the variables. We denote by $\sigma P$ the image of a polynomial $P$ under the action of a permutation $\sigma$.

\begin{definition}
An ideal $I\subset\K[X_1, \ldots, X_n]$ is called a \emph{symmetric ideal} if for every $P$ in $I$ and for every $\sigma \in \Sym_n$, $\sigma P$ belongs to $I$.
\end{definition}{}

Note that if $I$ is a symmetric ideal, then the variety $V(I)$ is closed under the action of $\Sym_n$ on the coordinates.

\subsection*{Specht Polynomials}

The so-called \emph{Specht polynomials} will play a central role in our proofs. Those polynomials were originally designed by Specht \cite{specht} to construct the different irreducible representations of $\Sym_n$. 

\begin{definition}Let $n\in \N$.
\begin{enumerate}
\item For a set $S=\{i_1, \ldots, i_r \}\subset \{1, \ldots, n\}$, we define the \emph{Vandermonde determinant} $\Delta(S)$ of the variables $X_i$, for $i\in S$:
$$ \Delta(S):=\prod_{1\leq j< k \leq r }(X_{i_j}-X_{i_k}).$$
\item Let $\lambda\vdash n$ and $T$ be a $\lambda$-tableau. Then the \emph{Specht polynomial} associated with $T$ is the polynomial
$$\spe_{T}:=\prod_{c}\Delta(T_{\cdot,c}),$$
where $c$ runs through the columns of $T$, and $T_{\cdot,c}$ denotes the entries in the $c$th column. We will say that $\spe_T$ is a Specht polynomial of shape $\lambda$.
\end{enumerate}
\end{definition}

\begin{example}
The Specht polynomial associated with the Young tableau
\[
\begin{Young}
$4$ & $2$ & $6$ & $1$ \cr
$8$ & $7$ & $5$  \cr
$3$  \cr
\end{Young}
\]
is 
\[
(X_4-X_8)(X_4-X_3)(X_8-X_3)(X_2-X_7)(X_6-X_5).
\]
\end{example}

\section{Zeros of Specht polynomials}\label{sec:Specht}
Throughout the paper, we will be interested in the ideal generated by all the Specht polynomials of a given shape, as well as in the corresponding variety. 
\begin{definition}Let $\K$ be a field, $n$ be an integer, and $\mu\vdash n$.
\begin{itemize}
    \item The \emph{$\mu$-Specht ideal}, denoted by $I^{\spe}_\mu$, is the ideal of $\K[X_1,\ldots,X_n]$ generated by all the Specht polynomials of shape $\mu$.
    \item We denote by $V_\mu$ the set of common zeros of all Specht polynomials of shape $\mu$, that is 
$$V_\mu:=V(I^{\spe}_\mu)=\bigcap_{\sh(T)=\mu} V(\spe_T).$$
\end{itemize}{}
 
\end{definition}

We aim at describing more precisely those varieties. Particular cases of these varieties  have been studied for example in \cite{yanagawa2019specht,watanabe2017vandermonde,froberg2016vandermonde}. Let us start with a few remarks. Let $T$ be a Young tableau, of shape $\mu \vdash n$. A point $x=(x_1,\ldots,x_n)$ is a zero of $\spe_T$ if and only if there exists a column of $T$ containing two indexes $i\neq j$ such that $x_i=x_j$. Following this easy observation, we get a characterization of $V_\mu$ that will be useful later: a point $x\in \K^n$ \emph{does not } belong to $V_\mu$ if and only if one can fill in a Young tableau of shape $\mu$ with the coordinates of $x$ in such a way that in every column, all the values are distinct.

This remark already implies some properties about varieties associated with Specht ideals, that will be useful in the following:

\begin{proposition}\label{prop:SpechtVar}
Let $x$ be a point in $\K^n$, and $\Lambda(x)$ its orbit type. Then:
\begin{enumerate}
\item[i)] The point $x$ is not in the variety $V_{\Lambda(x)}$,
\item[ii)] If $\mu$ is a partition of $n$ such that $x \notin V_\mu$, then $\Lambda(x) \trianglelefteq\mu$. 
\end{enumerate}
\end{proposition}
\begin{proof}
Let $\lambda=\Lambda(x)=(\lambda_1,\ldots,\lambda_r)$, and let $u_1,\ldots, u_r$ be the distinct coordinates of $x$, where for each $1\leq i \leq r$, $u_i$ appears $\lambda_i$ times.

Let us first prove i). Let $T$ be any Young tableau of shape $\lambda$ such that the indexes in the $i$th row are the indexes $j$ such that $x_j=u_i$. Then $\spe_T(x)\neq 0$, and hence $x$ is outside $V_\lambda$.

Now let us prove ii). Suppose $x$ not in $V_\mu$. This means that there exists a Young tableau $U$ of shape $\mu$ such that $\spe_U(x)\neq 0$. Thus we can fill $U$ with $u_1,\ldots,u_r$ in such a way that for every $1\leq i \leq r$, $u_i$ appears at most once per column, and we may assume without loss of generality that in every column, if $u_j$ is below $u_i$, then $i<j$. As a consequence, for every $1\leq k \leq r$, the $u_i$ for $i\leq k$ are contained in the first $k$ rows of $T$. This means that
\[
\lambda_1 + \ldots + \lambda_k \leq \mu_1 + \ldots + \mu_k,
\]
which implies $\lambda \trianglelefteq \mu$.

\end{proof}

Now we want to prove that $V_\lambda$ contains exactly the $V_\mu$ such that $\mu \trianglerighteq \lambda$. This is a consequence of the corresponding inclusion of ideals:

\begin{theorem}\label{thm:inclusionV_lamda}
Let $\lambda$ and $\mu$ be two partitions of $n$. Then the following assertions are equivalent:
\begin{enumerate}
\item[i)] The partition $\mu$ dominates $\lambda$, \emph{i.e.} $\lambda \trianglelefteq \mu$,
\item[ii)] The ideal $I^{\spe}_\mu$ contains $I^{\spe}_\lambda$, \emph{i.e.} $I^{\spe}_\lambda \subset I^{\spe}_\mu$,
\item[iii)] The variety $V_\lambda$ contains $V_\mu $, \emph{i.e.} $V_\mu \subset V_\lambda$.
\end{enumerate}  
\end{theorem}

\begin{remark}\label{rem:specht}
Note that the inclusion of ideals was not a priori equivalent to the inclusion of varieties, since it is not known whether Specht ideals are radical, see \cite{yanagawa2019specht}.  
\end{remark}{}

\begin{proof}
We start by proving i) implies ii). Let $\lambda=(\lambda_1, \ldots, \lambda_t)$, and $\mu$ such that $\lambda \trianglelefteq \mu$. 
We only need to consider the particular case where $\mu$ is of the form
\[
\mu=(\lambda_1, \ldots, \lambda_{i-1}, \lambda_{i}+1, \lambda_{i+1}, \ldots, \lambda_{j-1}, \lambda_{j}-1, \lambda_{j+1}, \ldots, \lambda_t),
\]
where $\lambda_{i-1} > \lambda_i$ and $\lambda_{j} > \lambda_{j+1}$, in order to ensure that $\mu$ is a partition. Indeed, it is known that we can go from $\lambda$ to any $\mu \trianglerighteq \lambda$ by a finite number of such steps, see e.g.~\cite[Prop. 2.3]{B73}. 

Let $T$ be a Young tableau of shape $\lambda$. We need to show that $\spe_T$ belongs to the ideal generated by all the polynomials $\spe_U$ where $U$ runs through all the Young tableaux of shape $\mu$. If $U$ is a Young tableau of shape $\mu$, then its columns have the same number of elements than $T$, except for two of them. Let $U_1$ and $U_2$ be these columns, and let $T_1$ and $T_2$ be the corresponding columns in $T$. If $a=|T_1|$ and $b=|T_2|$, then $|U_1|=a-1$, $|U_2|=b+1$ and $a-b\geq 2$.  By restricting our attention to these two columns, it is enough to prove, up to permutation, that the polynomial 
\[
P=\Delta(\{1,\ldots,a\})\Delta(\{a+1,\ldots,a+b\})
\]
is in the ideal $I$ of $\K[X_1,\ldots,X_{a+b}]$ generated by all the polynomials of the form
\[
\Delta(S_1)\Delta(S_2), \text{ with } |S_1|=a-1, |S_2|=b+1, \text{ and } S_1 \cup S_2 = \{1,\ldots,a+b\}.
\]
Consider the polynomial
\[
Q = \Delta(\{2,\ldots,a\})\Delta(\{1,a+1,\ldots,a+b\}) X_1^{a-b-1},
\]
and the polynomial 
\[
R=\sum_{i=1}^a \epsilon((1,i)) (1,i) Q,
\]
where $\epsilon$ denotes the signature. 
By construction, $R$ is in the ideal $I$. We need to prove that for any pair $1\leqslant \alpha < \beta \leqslant a$ and $a+1 \leqslant \alpha <\beta \leqslant a+b$, the polynomial $(X_\alpha - X_\beta)$ divides $R$, equivalently $R$ vanishes whenever $X_\beta = X_\alpha$. 
In other words, we want to show 

\[R_{\alpha,\beta}=R(X_1\ldots,X_{\beta-1},X_\alpha,X_{\beta+1},\ldots,X_n)=0\]
for every pair $1\leqslant \alpha < \beta \leqslant a$ and $a+1 \leqslant \alpha <\beta \leqslant a+b$ .

The latter case is obvious by definition of $Q$ and $R$. Suppose $1\leqslant \alpha < \beta \leqslant a$. Then for every $i \neq \alpha,\beta$ we have $((1,i)Q)_{\alpha,\beta}=0$, so that \[R_{\alpha,\beta} = \epsilon((1,\alpha))((1,\alpha)Q)_{\alpha,\beta} + \epsilon((1,\beta))((1,\beta)Q)_{\alpha,\beta}.\]
Now, we have to distinguish two cases, namely $\alpha=1$ and $\alpha \neq 1$. In the first case, we have \[R_{1,\beta} = Q_{1,\beta} -((1,\beta)Q)_{1,\beta} = 0\] while in the second case, the signatures are both negative, but the polynomial differ from each other by interchanging the variables $X_1$ and $X_\alpha$ in places $\alpha$ and $\beta$, and by definition of $Q$, these polynomials are opposite of each other, that is, \[R_{\alpha,\beta}=0.\]

This implies that $P=\Delta(\{1,\ldots,a\})\Delta(\{a+1,\ldots,a+b\})$ divides $R$. Since the degree of $R$ satisfies
\begin{align*}
\deg(R)\leq \deg(Q) &= \frac{(a-1)(a-2)}{2} + \frac{(b+1)b}{2} + a-b-1 \\ &= \frac{a(a-1)}{2} + \frac{b(b-1)}{2}  \\& = \deg(P),
\end{align*}
it implies that
\[
R=cP
\]
with $c\in \K$, and we need to check that $c$ is not zero. The degree of $Q$ in the variable $X_{1}$ is $b+ a-b-1=a-1$ while its degree in the variable $X_i$ for $2\leqslant i \leqslant a $ is $a-2$. Thus, the degree of $R$ in the variable $X_1$ is exactly $a-1$ and the corresponding coefficient is $\Delta(\{2,\ldots,a\})\Delta(\{a+1,\ldots,a+b\})$, which is also the coefficient of $X_1^{a-b}$ in $P$. Hence $R = P$. 

Since ii) obviously implies iii), we only have to prove that iii) implies i). 
Assume then that $V_\mu \subset V_\lambda$. Consider $x$ with orbit type $\Lambda(x) = \lambda$. Then, according to i) of Proposition \ref{prop:SpechtVar}, $x\notin V_\lambda$. Then by assumption, $x\notin V_\mu$, and ii) of Proposition \ref{prop:SpechtVar} yields $\lambda \trianglelefteq \mu$.

\end{proof}

Finally, as a consequence of the previous results, we get a characterization of $V_\mu$ in terms of orbit types:

\begin{corollary}\label{thm:CaracSpechtHlambda}
The set of common zeros of Specht polynomials associated with Young tableaux of given shape $\mu$ can be characterised as    \[V_\mu= \left(\bigcup_{\lambda\trianglelefteq \mu} H_{\lambda}\right)^{c}  =\bigcup_{\nu\not\trianglelefteq \mu} H_{\nu}\]
\end{corollary}

\begin{proof}
We want to show $V_\mu^c= \bigcup_{\lambda\trianglelefteq \mu}H_{\lambda}$.
The direct inclusion is nothing but ii) in Proposition~\ref{prop:SpechtVar}.
Conversely, let $x\in H_\lambda$, with $\lambda \trianglelefteq \mu$.
According to i) in Proposition \ref{prop:SpechtVar}, $x$ is outside $V_\lambda$. 
Since $\lambda \trianglelefteq \mu$, it follows from Theorem~\ref{thm:inclusionV_lamda} that $V_\mu \subset V_\lambda$, so that $x\in V_\mu^c$.

\end{proof}

\section{Specht polynomials in symmetric ideals}\label{sec:Main}
In this section we show that if a symmetric ideal contains polynomials with sparse leading component, then this ideal will contain many Specht polynomials. Let us be more precise: First, to every monomial we associate a partition:

\begin{definition}\label{def:Mu(m)}
Let $m$ be a monomial of weight $l$ and degree $d$ in $\K[X_1, \ldots, X_n]$. The partial degrees of $m$ induce a partition of $d$ of length $l$, say $(\lambda_1, \ldots, \lambda_l)$. 

If moreover we assume that $l + d \leq n$, we can define a partition $\mu(m)$ of $n$ by
\[
\mu(m)=(\lambda_1 +1, \lambda_2 + 1, \ldots, \lambda_l + 1, \underbrace{1, \ldots, 1}_{n-d-l}).
\]
\end{definition}

\begin{example}\label{exa1}
Let $n=12$, and 
\[
m=X_2X_4^4X_5^2.
\] 
Then 
\[
\mu(m)=(5,3,2,1,1).
\]
\end{example}

Then we will show:

\begin{theorem}\label{thm:SpechtTous}
Let $I \subset \K[X_1,\ldots,X_n]$ be a symmetric ideal. Assume that there exists $P \in I$ of degree $d$, such that $d + \wt(P_d)\leqslant n$. Then, for every  monomial $m \in \Mon(P_d)$, the ideal $I$ contains every $\spe_T$ 
for which  $\sh(T) \trianglelefteq \mu(m)^\perp$.
In other words: 
\[
I^{\spe}_\lambda \subset I
\]
for every $\lambda \trianglelefteq \mu(m)^\perp$. 
\end{theorem}

According to Theorem~\ref{thm:inclusionV_lamda}, it is enough to prove that $I$ contains the Specht polynomials of shape $\mu(m)^\perp$. Hence we only need to prove:

\begin{proposition}\label{thm:Specht}
Let $I \subset \K[X_1,\ldots,X_n]$ be a symmetric ideal. Assume that there exists $P \in I$ of degree $d$, such that $d + \wt(P_d)\leqslant n$. Then, for every  monomial $m \in \Mon(P_d)$, the ideal $I$ contains every $\spe_T$ 
for which  $\sh(T) = \mu(m)^\perp$.
In other words $I^{\spe}_{\mu(m)^\perp} \subset I$.

\end{proposition}

In the following proof, we will assume that the characteristic of $\K$ is zero. This allows for more conceptual proof. 
This assumption on the characteristic ensures that the factorials in the end of the proof do not vanish.
We provide a general proof in the Appendix.
\begin{proof} Let $\K$ be a field of characteristic 0.
Since the ideal $I$ is symmetric, we may assume that \[m=X_1^{k_1}X_2^{k_2}\cdots X_l^{k_l}.\] and that $\Supp(P_d)=\{1,\ldots,\wt(P_d)\}$. Its associated partition is \[\mu:=\mu(m)=(k_1+1,k_2+1,\ldots,k_l+1,\underbrace{1,\ldots,1}_{n-d-l}).\] The statement says that the ideal $I$ contains any Specht polynomial of the form \[\Delta_1\Delta_2\cdots\Delta_l,\] where  the $\Delta_i$ are Vandermonde polynomials  in disjoint sets of variables, each of size $k_i+1=\mu_i$. Thanks to its symmetry, it is enough to show that $I$ contains one such polynomial. 

Our strategy consists in using  $X_i^{k_i}$ to build a Vandermonde polynomial involving $X_i$ and $k_i$ variables that do not appear in $P_{d}$. Our assumption on $P_d$ guarantees that there are enough free variables to do so. 

More precisely, we can take $I_1, \ldots, I_l$, disjoint subsets of $\{1,\ldots,n\}$ such that for any $1\leq i \leq l$, there are $k_i$ elements in $I_i$, and none of them appears in $P_d$. Let, for $1\leq i \leq l$,
\[
J_i = \{i\} \cup I_i.
\]
We will show that there exist polynomials $R_\sigma \in \K[X_1,\ldots,X_n]$, for $\sigma \in \Sym_n$ such that:
\[
\Delta(J_1)\cdots\Delta(J_l) = \sum_{\sigma \in \Sym_n} R_\sigma \sigma P.
\]
Here, applying the strategy used to prove Theorem~\ref{thm:inclusionV_lamda} we give explicit polynomials $R_\sigma$ when the characteristic of $\K$ is $0$. In the general case, we can give a recursive construction of these polynomials; we postpone this construction to Appendix~\ref{App:A}.
Consider the polynomials
\[
Q=\Delta(I_1)\cdots\Delta(I_l)P
\]
and
\[
R=\sum_{\sigma \in \Sym_{J_1}\times \cdots \times\Sym_{J_l}} \epsilon(\sigma) \sigma Q,
\]
where $\Sym_{J_1}\times \cdots \times\Sym_{J_l}$ is seen as a subgroup of $\Sym_n$.
By construction, for every $\rho \in \Sym_{J_1}\times \cdots \times\Sym_{J_l}$, 
\[
\rho R = \epsilon(\rho) R
\]
so that $\Delta(J_1)\cdots\Delta(J_l)$ divides $R$. Furthermore, since 
\begin{align*}
\deg (Q) & = d + \sum_{i=1}^l \frac{k_i(k_i-1)}{2} \\ & = \sum_{i=1}^l \frac{k_i(k_i-1)}{2} + k_i
\\ & = \sum_{i=1}^l \frac{k_i(k_i+1)}{2} 
\\ & = \deg (\Delta(J_1)\cdots\Delta(J_l)),
\end{align*}
we get 
\[
R = c \Delta(J_1)\cdots\Delta(J_l)
\]
with $c\in \K$. In order to check that $c$ is not $0$, we look at the coefficent of $R$ corresponding with $m=X_1^{k_1}\cdots X_l^{k_l}$, seen as an element of $(\K[X_{l+1},\ldots,X_n])[X_1,\ldots,X_l]$. 

In $Q$, this coefficient is $\Delta(I_1)\cdots\Delta(I_l)$. If, for $1\leq i \leq l$, the permutation $\sigma \in \Sym_{J_1}\times \cdots \times\Sym_{J_l}$ does not let $i$ invariant, the assumption on $P_d$ ensures that the coefficient of $m$ in $\sigma Q$ will be $0$. Therefore, the coefficient of $R$ corresponding with $m$ is 
\[
\sum_{\sigma \in \Sym_{I_1}\times \cdots \times\Sym_{I_l}} \epsilon(\sigma) \sigma \Delta(I_1)\cdots\Delta(I_l) =  k_1!\cdots k_l!\Delta(I_1)\cdots\Delta(I_l)
\]
and hence $R =  k_1!\cdots k_l! \Delta(J_1)\cdots\Delta(J_l)$.

\end{proof}

\section{Applications}\label{sec:appl}

\subsection{Computing points in symmetric varieties}\label{ssec:DegreePrinciple}

Let $n$ be an integer and $I$ be a symmetric ideal in $\K[X_1,\ldots,X_n]$. What can we say about the variety $V(I)$? For instance, can we algorithmically  decide if $V(I)$ is empty in an efficient manner making use of the structure of $I$?

Over any real closed field $\R$ the so-called \emph{half-degree principle} \cite{timofte2003positivity} can be used to simplify the algorithmical task of root finding.  This statement says (\cite[Corollary 1.3]{riener2012degree}):

\begin{theorem}\label{thm:HalfDegreeCordian}
Let $\K$ be a real closed field, and let $P$ be a symmetric polynomial in $\K[X_1,\ldots,X_n]$ of degree $d$, and let $k=\max(2,\lfloor \frac{d}{2} \rfloor )$. Then there exists $x \in \K^n$ such that $P(x)=0$ if and only if there exists $y\in \K^n$ with at most $k$ distinct coordinates such that $P(y)=0$.
\end{theorem}

This implies the following result on symmetric ideals:
\begin{corollary}\label{cor:degreeprinciple}
Let $\K$ be a real closed field, and let $I$ be a symmetric ideal of $\K[X_1,\ldots,X_n]$, generated by $P_1,\ldots,P_l$. 
Let $d=\max(\deg(P_1), \ldots, \deg(P_l))$. Then $V(I)$ is non empty if and only if it contains a point $x\in \K^n$ with at most $d$ distinct coordinates.
\end{corollary}
\begin{proof}
Over a real closed field the variety $V(I)$ is exactly the variety defined by 
\[
Q=\sum_{i=1}^l \sum_{\sigma\in \Sym_n} \sigma(P_i)^2
\]
and we can apply Theorem~\ref{thm:HalfDegreeCordian}.
\end{proof}

The algorithmic implications of this result are the following. 
Take $\mu\vdash n$ of length $d$. For every polynomial $P\in \K[X_1,\ldots, X_n]$ we consider $$P^\mu:=P(\underbrace{Z_1,\ldots, Z_1}_{\mu_1},\underbrace{Z_2,\ldots, Z_2}_{\mu_2},\ldots, \underbrace{Z_d,\ldots, Z_d}_{\mu_d})$$ and denote by $I^\mu\subset \K[Z_1,\ldots,Z_d]$ the resulting ideal in $d$ variables.
Moreover consider the topological closure $\overline{H}_\mu$ of $H_\mu$ and the map
\[\Phi_\mu:\,V(I^\mu)\longrightarrow (V(I)\cap \overline{H}_\mu)/S_n   
\]
which associates to a point  $x=(x_1,\ldots,x_d)\in V(I^\mu)$ the $S_n$-orbit of the point
$$x=(\underbrace{x_1,\ldots, x_1}_{\mu_1},\underbrace{x_2,\ldots, x_2}_{\mu_2},\ldots, \underbrace{x_d,\ldots, x_d}_{\mu_d}).$$ This map is clearly surjective, and from the natural decomposition 
\[
V(I) = \bigcup_{\mu \vdash n} (V(I) \cap H_\mu) = \bigcup_{\nu \vdash n} (V(I) \cap \overline{H_\nu})
\]
we thus get
\[
V(I)/S_n = \bigcup_{\nu \vdash n} \Phi_\nu(V(I^\nu)).
\]
Then Corollary~\ref{cor:degreeprinciple} says precisely that $V(I)$ is empty if and only if $V(I^{\mu})$
is empty for every partition $\mu$ of $n$ with $\len(\mu) \leq d$. 
  Since the number of $d$-partitions of $n$ is bounded by $(n+1)^d$ the original problem in $n$ variables reduces to a polynomial number of problems in $d$ variables.

Our results yield a stronger version of this principle, under additional assumption on the support of the polynomials: On the one hand our results are valid for any field, and on the other hand, not only our varieties contain points with few distinct coordinates, but they contain \emph{only} points with few distinct coordinates.
More precisely, Theorem~\ref{thm:SpechtTous} gives, in this context:

\begin{theorem}\label{thm:main}
Let $I\subset \K[X_1,\ldots X_n]$ be a symmetric ideal. Assume that there exists $P\in I$ of degree $d$ such that $\wt(P_d) + d \leqslant n$ and let $m \in Mon(P_d)$. Then  
$$V(I)\cap H_\lambda =\emptyset\, \text{  for all } \lambda \trianglelefteq \mu(m)^\perp.$$
In other words, 
\[
V(I)/S_n = \bigcup_{\nu\not\trianglelefteq \mu(m)^{\perp}}\Phi_\nu(V(I^\nu)).
\]
\end{theorem}

\begin{proof}
According to Proposition~\ref{thm:Specht}, the variety $V(I)$ is contained in the variety $V_{\mu^\perp}$ associated with the Specht ideal $I_{\mu^\perp}^{\spe}$. Corollary \ref{thm:CaracSpechtHlambda} then yields
\[
V(I) \subset \bigcup_{\lambda\not\trianglelefteq \mu(m)^{\perp}}H_\lambda.
\]
Since $\K^n$ is the disjoint union of the subsets $H_\lambda$, it follows that for all $\lambda$ with $\lambda \trianglelefteq \mu(m)^\perp$, $V(I)\cap H_\lambda = \emptyset$. 
Furthermore, we can write 
\[
V(I) = \bigcup_{\lambda\not\trianglelefteq \mu(m)^{\perp}} (V(I)\cap H_\lambda).
\]
Thus, to prove the second part of the statement, it is enough to prove that 
\[
\bigcup_{\lambda\not\trianglelefteq \mu(m)^{\perp}} (V(I)\cap H_\lambda) = \bigcup_{\nu\not\trianglelefteq \mu(m)^{\perp}} (V(I)\cap \overline{H_\nu}).
\]
One inclusion is trivial, we focus on the other one. Assume that $x \in \overline{H_\nu}$. Then naturally, $\nu \trianglelefteq \Lambda(x)$. So if $\nu\not\trianglelefteq \mu(m)^{\perp}$, we also have $\Lambda(x) \not\trianglelefteq \mu(m)^{\perp}$.
\end{proof}

Hence, if one is able to compute the points in the variety $V(I^{\nu})$, one gets all the points of $V(I)$. Also note that the length of the partitions $\nu$ is at most $d$, this comes from the following observation:

\begin{proposition}Let $n$ be an integer, and $m$ be a monomial of degree $d$, with $d + \wt(m) \leq n$. Then for every  partition $\lambda$ of $n$ such that $\len(\lambda) > d $,
\[
\mu(m)^\perp \trianglerighteq \lambda.
\]
\end{proposition}

\begin{proof}
The proof consists in two steps. First we prove that
\[
\mu(m)^\perp \trianglerighteq (n-d,\underbrace{1,\ldots,1}_{d}).
\]
Indeed, if $m=X_1^{k_1}\cdots X_l^{k_l}$, 
\[
\mu(m)=(k_1+1,k_2+1,\ldots,k_l+1,\underbrace{1,\ldots,1}_{n-d-l})
\]
has length $n-d$, so that 
\[
\mu(m)^\perp=(n-d,\ldots)\trianglerighteq (n-d,\underbrace{1,\ldots,1}_{d}).
\]
Second, if $\len(\lambda)>d$, then  
\[
(n-d,\underbrace{1,\ldots,1}_{d}) \trianglerighteq \lambda.
\]
Indeed, if not, there exists $j$ such that 
\[
\sum_{i=1}^j \lambda_i > (n-d) + (j-1).
\] 
In this case, since $\len(\lambda) \geq d+1$,
\begin{align*}
n  = \sum_{i=1}^{\len(\lambda)} \lambda_i  & \geq \sum_{i=1}^j \lambda_i + \len(\lambda)-j
\\ & > n - d + j -1 + \len(\lambda)-j
\\ & > n.
\end{align*}

\end{proof}
\begin{remark}
The proposition above shows that we can always ensure that every point of the variety $V(I)$ has at most $d$ distinct coordinates. 
If the monomial $m$ is $X_1^d$ then $\mu(m)^\perp=(n-d,1,\ldots,1)$ and this is the only case where we need to consider all $d$-partitions of $n$.  

Already for the monomial $m=X_1^{d-1}X_2$, we have $\mu(m) = (d,2,1, \ldots,1) $, and $\mu(m)^\perp=(n-d,2,1,\ldots,1)$.
Then for every $2\leq k \leq (n-d-2)/2$, the partition $(n-d-(k-2), k, 1, \ldots, 1)$ has length $d$ and is dominated by $\mu(m)^\perp$: they are among the partitions that we do not need to consider.

The most favourable case will be $m=X_1X_2 \cdots X_d$, where $\mu(m)^\perp = (n-d,d)$. In this case, the only partitions $\lambda=(\lambda_1, \ldots, \lambda_t)$ that are not dominated by $\mu(m)^\perp$ are the ones with $\lambda_1 > n-d$. 

Therefore, a fine analysis on the actual monomials of highest degree in the generators of $I$ allows to reduce the number of partitions that need to be considered, which can be seen as a stronger version of the degree principle. 
\end{remark}
One further natural consequence is that our variety is contained in a finite union of $d$-dimensional subspaces, hence:
\begin{corollary}
Let $I \subset \K[X_1,\ldots,X_n]$ be a symmetric ideal. Assume that there exists $P \in I$ of degree $d$, such that $d + \wt(P_d)\leqslant n$. Then the dimension of the variety $V(I)$ is at most  $d$. 
\end{corollary}
In a more general setup, Nagel and R\"{o}mer \cite{nagelroemer} study sequences of symmetric ideals.  They show in particular that the dimension of the ideals they study is a linear function in $n$. In our more restricted framework, we thus obtain a stabilization of the dimension of such sequences.  

\subsection{Isotypic components of symmetric ideals}\label{ssec:isotypic}
 The action of the symmetric group $\Sym_n$ on $\K[X_1,\ldots,X_n]$ is linear, giving the polynomial ring the structure of a $\K[\Sym_n]$-module. If we assume that the characteristic of $\K$ is $0$, then every   $\K[\Sym_n]$-module can be decomposed as a direct sum of irreducible submodules.  It is well known (see \textit{e.g.} \cite{S01}) that the irreducible  $\K[\Sym_n]$-modules are in correspondence with the partitions of $n$. These modules are called Specht modules, denoted by $S^\lambda$. It follows that every $\K[\Sym_n]$-module $U$ can be uniquely written as
\[
U\simeq \bigoplus_{\lambda \vdash n} U_\lambda, 
\]
where for every partition $\lambda$ of $n$, $U_\lambda$ is a direct sum of irreducible submodules isomorphic to $S^\lambda$. The submodule $U_\lambda$ is called the \emph{$\lambda$-isotypic component} of $U$. 

Now let $I\subset \K[X_1,\ldots,X_n]$ be a symmetric ideal. It is also a $\K[\Sym_n]$-module, and we have, for every partition $\lambda$ of $n$,
\[
I_\lambda = \K[X_1, \ldots, X_n]_\lambda \cap I.
\]

Let $\lambda$ be a partition of $n$, then the linear subspace $W_\lambda$ of $\K[X_1, \ldots, X_n]$ generated by all the Specht polynomials of shape $\lambda$ is an irreducible submodule of $\K[X_1, \ldots, X_n]_\lambda$ isomorphic to $S^\lambda$. For any other irreducible submodule $\widetilde{W}_\lambda$ in $\K[X_1, \ldots, X_n]_\lambda$, we have an isomorphism $\varphi$ between $W_\lambda$ and $\widetilde{W}_\lambda$. Let $T$ be a Young tableau of shape $\lambda$. Since $\varphi$ respects the action of $\Sym_n$, for any $\tau$ transposition of two elements in a same column of $T$,
\[
\tau \varphi(\spe_T) = - \varphi(\spe_T),
\]
so that $\varphi(\spe_T)$ has to be divisible by $\spe_T$. It follows that $\K[X_1, \ldots, X_n]_\lambda$ is included in the Specht ideal $I^{\spe}_\lambda$, and therefore Theorem~\ref{thm:SpechtTous} gives:

\begin{theorem}\label{thm:IdealVersion}
Let $\K$ be a field of characteristic $0$ and $I \subset \K[X_1,\ldots,X_n]$ be a symmetric ideal. Assume that there exists $P \in I$ of degree $d$, such that $d + \wt(P_d)\leqslant n$. Then for every $m \in \Mon(P_d)$, for every $\lambda \trianglelefteq \mu(m)^\perp$, the ideal $I$ contains the isotypic component $\K[X_1, \ldots, X_n]_\lambda$. In other words
\[
I_\lambda = \K[X_1, \ldots, X_n]_\lambda,
\]
or equivalently
\[
\left(\K[X_1, \ldots, X_n]/I\right)_\lambda = \{ 0 \}.
\]
\end{theorem}

Given polynomials  $P_1,\ldots,P_l\in \K[X_1,\ldots,X_{n_0}]$, they naturally induce a symmetric ideal in $\K[X_1,\ldots,X_n]$ for any $n\geq n_0$. We get an increasing sequence of symmetric ideals $(I)_{n\geq n_0}$, and one can study stabilization properties in terms of representations (see for instance \cite{sam2015stability,church2015fi}).

We remark that when $n$ is large enough, the condition on the support of the leading component of $P$ is automatically fulfilled and we can immediately deduce:

\begin{theorem}\label{Asymptotic}
Let $\K$ be of characteristic $0$ and $n_0$ be an integer. Given $Q_1, \ldots, Q_l $ in $\K[X_1,\ldots,X_{n_0}]$, consider for any $n\geq n_0$ the ideal
\[
I_n = \langle \sigma(Q_i), \sigma\in \Sym_{n}, 1\leq i \leq l \rangle.
\]
Then if $n$ is large enough, for every $1\leq i \leq l$, every monomial $m$ of $Q_i$ of maximal degree, and any $\lambda$ partition of $n$ such that $\lambda \trianglelefteq \mu(m)^\perp$,
\[
\left(\K[X_1, \ldots, X_n]/I_n\right)_\lambda = \{ 0 \}.
\]
\end{theorem}{}

\subsection{Symmetric sums of squares on symmetric varieties}\label{ssec:SOS}

As a final application we consider sums of squares of real polynomials. Let $P\in\K[X_1,\ldots,X_n]$. Then $P$ is called a sum of squares, if there exist polynomials $P_1,\ldots,P_k$ with
$$P=P_1^2+\ldots+P_k^2.$$ Sum of squares are the  cornerstone in 
the so called moment approach to polynomial optimization \cite{lasserre}: in general,  it can be  decided by semidefinite programming if a given polynomial affords a decomposition as a sum of squares. 
The case of symmetric sums of squares has received some interest by different authors \cite{blrie,goel2016choi,raymond2018,riener2013,kurpisz2016sum}. 

In \cite{blrie}, Blekherman and the second author described how to characterize symmetric sums of squares through  representation theory. 
More precisely, they use the theory of higher Specht polynomials \cite{terasoma1993higher} to construct, for every $\lambda\vdash n$, a square matrix polynomial $Q^\lambda$ of size $s_\lambda=\dim(S^\lambda)$, whose entries are symmetric polynomials. Furthermore, these entries are products and sums of elements in $\R[X_1,\ldots, X_n]_\lambda$. So even though they are symmetric, they belong to the ideal generated by the Specht polynomials of shape $\lambda$.  These matrices can be used to show that every symmetric polynomial $P$ that is a sum of squares can be written in the form
\begin{equation*}
P=\sum_{\lambda\vdash n} \Tr( P^\lambda\cdot Q^\lambda),
\end{equation*}
where  each $P^\lambda\in\R[X_1,\ldots, X_n]^{s_\lambda \times s_\lambda}$ is a sum of symmetric squares matrix  polynomial, i.e. $$P^\lambda=L^{t}L$$ for some matrix  $L$ whose entries are symmetric polynomials.
 
 Since the $\lambda$-Specht ideal contains all the coefficients of  $Q^\lambda$ we can apply  Theorem \ref{thm:IdealVersion} to obtain the following result on representations of a symmetric polynomial modulo a symmetric ideal. 
 
 \begin{theorem}\label{sos1}
 Let $P\in\R[X_1,\ldots,X_n]^{S_n}$ be a symmetric sum of squares polynomial and $I$ be a symmetric ideal in $\R[X_1,\ldots,X_n]$. Further, we assume that there exists $F \in I$ of degree $d$, such that $d + \wt(F_d)\leqslant n$. 
 Then $P$ can be written as
 
 \[
 P=\sum_{\nu\not\trianglelefteq \mu(m)^{\perp}} \Tr( P^\lambda\cdot Q^\lambda) \mod I,
 \]
 where again each $P^\lambda=L^{t}L$ for some matrix polynomial $L$ whose entries are symmetric polynomials.
 \end{theorem}
 A case of special interest is the case when $I$ is the gradient ideal $I_{\text{grad}}(P)$ of a given polynomial $P$ of even degree $2d$. Sturmfels and Nie \cite{nie} showed that a polynomial that is positive on its gradient variety $V(I_{\text{grad}})$ can always be written as a sum of squares modulo its gradient ideal. When $P$ is a symmetric polynomial our results can be applied to reduce the problem size. 
 
 It is worth remarking that perturbations can be used to transfer a polynomial into the situation of finitely many critical points in a symmetric way:
 For example, Hanzon and Jibetean \cite{han}, as well as Jibetean and  Laurent \cite{laurent} considered
the following perturbation of a polynomial:
\[
P_\varepsilon:=\varepsilon\cdot (X_1^{2d+2}+\ldots+ X_n^{2d+2})+P.
\]
Since the perturbation term is positive definite and of higher degree,  the perturbed polynomial $P_\epsilon$ has a global minimum and the minimal value converges to the infimum of $P$ with $\varepsilon \to 0$. Moreover, if $P$ in fact has global minimizers, each connected component of the set of global minimizers of $P$ contains a point which is limit of a branch of local  minimizers of $P_\varepsilon$ (see \cite{laurent}). 
 Furthermore, observe that  the quotient $I_{\text{grad}}(P_\varepsilon)$ is generated by the  polynomials  $2(d+1) \varepsilon X_i^{2d+1}+\frac{\partial P}{\partial X_i}$, for $1\leq i\leq n$ and thus the quotient  $\R[X_1,\ldots, X_n]/I_{\text{grad}}(P_\varepsilon)$  is a finite dimensional vector space. 
This structure of the symmetric gradient ideal gives us further restrictions on both the number and the sizes of the matrices in Theorem \ref{sos1}. Indeed, it follows that only irreducible representations corresponding to partitions with at most $2d+1$ rows can appear.  For a fixed $d$, the number of these partitions is polynomial in $n$ and it follows from \cite[Theorem 2.5]{basu2015isotypic} that the sizes of the above matrices are polynomial as well. In order to use Theorem \ref{sos1} practically to decide if a polynomial is a sum of squares of polynomials, one uses semi-definite programming. The complexity of such a program is mainly determined by the size of the matrices used to define it. Therefore, our discussion above can in fact be used to design efficient semi-definite programs, of a size which depends only polynomially on $n$, to check if $P_\varepsilon\geq 0$ for a given  $\varepsilon>0$. 

\section{Concluding comments and open questions}

This paper provides properties of the ideals generated by all the Specht polynomials associated with Young tableaux of a given shape. 
In particular, one of our main results shows how the inclusions of these ideals relates to the  comparision of the associated partitions in the dominance order. 
The algebraic geometric aspects of these ideals are of special interest.
It will be further research to see how the results presented here could be useful, for example to compute Gr\"{o}bner bases of Specht ideals, or to decide the radicalness of Specht ideals, conjectured in \cite{yanagawa2019specht}.

A second main result, Theorem~\ref{thm:SpechtTous}, emphasizes the connection between ideals inviariant under the action of the symmetric group and Specht ideals. 
This leads to several algorithmic applications, such as computing points in the corresponding varieties, or certifying the non-negativity of symmetric polynomials.
An analogue study would be of interest in a more general setup, in particular for other groups affording combinatorial descriptions similar to Specht polynomials. 

\subsection*{Acknowledgements}

This work has been supported by European Union's Horizon 2020 research and innovation programme under the Marie Sk\l{}odowska-Curie grant agreement 813211 (POEMA) and the Troms\o~ Research foundation grant agreement 17matteCR.

We would like to thank the anonymous referees for helpful suggestions. In particular, one of their questions led us to a stronger statement of Theorem~\ref{thm:inclusionV_lamda}. 

\bibliographystyle{abbrv}
\bibliography{.bib}

\appendix

\section{Proof of Proposition~\ref{thm:Specht} in general characteristic}\label{App:A}
Recall that we want to show that there exist polynomials $R_\sigma \in \K[X_1,\ldots,X_n]$, for $\sigma \in \Sym_n$ such that:
\[
\Delta(J_1)\cdots\Delta(J_l) = \sum_{\sigma \in \Sym_n} R_\sigma \sigma P.
\]

In order to avoid an overload of notation we rename the variables in the following way:
\[ Z_{i,s}\,\text{ for } 1\leq i\leq \len(\mu) \text{ and } 1\leq s\leq \mu_i,\]
where we identify $Z_{i,1}$ with $X_i$ for $1\leq i\leq \len(\mu)$. In this way whenever $s>1$ then $Z_{i,s}$ does not appear in $P_d$. In this setting, the monomial $m$ is written as
\[
m=Z_{1,1}^{k_1} Z_{2,1}^{k_2} \cdots Z_{l,1}^{k_l}.
\]
We will show the existence of polynomials $R_\sigma$, only involving the variables $Z_{i,s}$ for  $1\leq i \leq l$, such that
\[
\Delta(J_1)\cdots\Delta(J_l)=\prod_{i=1}^l\prod_{1 \leqslant s<t\leqslant \mu_i} (Z_{i,s}-Z_{i,t}) = \sum_{\sigma \in S_\mu} R_\sigma \sigma P,\] where 
\[S_\mu=S_{\mu_1}\times S_{\mu_2}\times\cdots\times S_{\mu_{l}}\times \underbrace{S_1 \times \cdots \times S_1}_{n-d-l}\] 
and each factor acts on the corresponding subset of variables, namely:
\[\sigma Z_{i,s} = Z_{i,\sigma_i(s)}.\]

Let us show this by induction on the degree $d$ of $P$. If $d=0$, there is nothing to prove.

Now assume $d>1$. Up to a rescaling, we may assume that \[P = m + S,\] where  $S$ is a polynomial of degree at most $d$ such that $m \not \in \Mon(S)$ and $S_d$ does not contain any variable $Z_{i,s}$ with $s>1$. 

Let $\tau$ be the transposition exchanging $Z_{l,1}$ and $Z_{l,\mu_l}$. Then \[P - \tau P = (Z_{l,1} - Z_{l,\mu_l}) Q\] where $Q$ is a polynomial of degree $d-1$. Because $Z_{l,\mu_l}$ does not appear in $S_d$, we can write \[Q= m'  + S'\] with $m'= \left(\prod_{i=1}^{l-1}Z_{i,1}^{k_i}\right) Z_{l,1}^{k_l-1}$ and $S'$ is a polynomial of degree at most $d-1$. Since the only new variable appearing in $Q_d$ is $Z_{l,\mu_l}$, we have \[
d-1+\wt(Q_d) \leq d + \wt(P_d) \leq n,
\]
in such a way that we can apply the induction hypothesis on $Q$. This provides polynomials $R_\rho'$, only involving the variables $Z_{i,s}$ for  $1\leq i \leq l$, except $Z_{l,\mu_l}$, such that
 \[ \left(\prod_{i=1}^{l-1}\prod_{1 \leqslant s < t \leqslant \mu_i}(Z_{i,s}-Z_{i,t})\right) \prod_{1 \leqslant s < t \leqslant \mu_l-1}(Z_{l,s}-Z_{l,t}) = \sum_{\rho \in S'} R'_\rho \rho Q,\]
where \[S' = S_{\mu_1}\times S_{\mu_2}\times\cdots\times \left(S_{\mu_{l}-1}\times S_1\right)\times \underbrace{S_1 \times \cdots \times S_1}_{n-d-l}\] can be seen as a subgroup of $S_\mu$.

Since any  $\rho \in S'$ leaves the product $\prod_{s=1}^{\mu_l-1}(Z_{l,s}-Z_{l,\mu_l})$ unchanged, we have \begin{eqnarray*} \prod_{i=1}^{l}\prod_{1 \leqslant s < t \leqslant \mu_i}(Z_{i,s}-Z_{i,t}) &=& \prod_{s=1}^{\mu_l-1}(Z_{l,s}-Z_{l,\mu_l}) \sum_{\rho \in S'}R'_\rho \rho Q \\&=& \sum_{\rho \in S'} R'_\rho \rho\left(\left(\prod_{s=1}^{\mu_l-1}(Z_{l,s}-Z_{l,\mu_l}) \right)Q\right)\\&=& \sum_{\rho \in S'} R'_\rho \prod_{s=2}^{\mu_l-1}(Z_{l,\rho_l(s)}-Z_{l,\mu_l}) \rho\left( P - \tau P\right).
\end{eqnarray*} 
Because $\rho \tau \in S_\mu$, we can rewrite this expression as $\sum_{\sigma \in S_\mu} R_\sigma \sigma P$ with the desired properties.

\end{document}